\documentclass[11pt]{article}

\usepackage{amsmath,amssymb,amsthm,amsfonts}
\usepackage{geometry}
\usepackage{hyperref}
\hypersetup{hidelinks,
 pdfauthor={Wei Xie},
 pdftitle={Minor-Order Exponent Profiles of Oscillatory Matrices}}

\theoremstyle{plain}
\newtheorem{theorem}{Theorem}[section]
\newtheorem{lemma}[theorem]{Lemma}
\newtheorem{corollary}[theorem]{Corollary}

\theoremstyle{definition}
\newtheorem{definition}[theorem]{Definition}
\newtheorem{example}[theorem]{Example}

\theoremstyle{remark}
\newtheorem{remark}[theorem]{Remark}

\numberwithin{equation}{section}

\newcommand{\cint}[2]{[#1,#2]_{\preceq}}

\title{Minor-Order Exponent Profiles of Oscillatory Matrices}
\author{Wei Xie}
\date{26-09-05}

\begin{document}
	
	\maketitle
	
\begin{abstract}
For an oscillatory matrix \(A\), let \(e_k(A)\) be the least positive
integer \(m\) for which every minor of order \(k\) of \(A^m\) is positive.
We determine the profile \(\mathbf e(A)=(e_1(A),\ldots,e_n(A))\).
	We prove that, for every nonsingular totally nonnegative matrix, the
positive entries in each row of a multiplicative compound form an interval
when the indexing subsets are ordered componentwise.  The two endpoints of
these intervals are order-preserving and compose under matrix
multiplication.  Consequently, all minors of a fixed order are positive if
and only if the two remote corner minors of that order are positive.  For an
oscillatory matrix, each \(e_k(A)\) is therefore the larger of the first
positivity times of the corresponding upper-right and lower-left corner
minors.

This description yields the sharp inequalities
\begin{align*}
 e_k(A)&\leq\max\{k,n-k\} &&(1\leq k<n),\\
 |e_{k+1}(A)-e_k(A)|&\leq1 &&(1\leq k\leq n-2).
\end{align*}
If \(D=\operatorname{diag}(1,-1,1,-1,\ldots)\), then
\(DA^{-1}D\) is again oscillatory and
\(e_k(DA^{-1}D)=e_{n-k}(A)\) for \(1\leq k<n\);
the determinant exponent remains one.

	The ordered positions of the positive factors in an adjacent bidiagonal
factorization determine the entire profile, independently of the positive
factor values.  Each one-sided factor sequence can be represented by a
single permutation and its southwest rank table.  This gives a finite
characterization of all profiles and an integer unimodular realization of
each one.  We further obtain a compatibility condition across minor orders:
if \(3\leq k\leq n-3\) and \(e_k(A)\leq2\), then
\[
e_j(A)
\leq
2+\left\lceil\frac{\lvert j-k\rvert}{2}\right\rceil,
\qquad
2\leq j\leq n-2.
\]
In particular,
\[
e_k(A)\leq2
\quad\Longrightarrow\quad
e_{k+2}(A)\leq3,
\qquad
3\leq k\leq n-4,
\]
and both the constant \(3\) and the range of \(k\) are sharp.
\end{abstract}
	

	
\section{Introduction}
\label{sec:introduction}

Throughout, all matrices are real and have order \(n\geq2\).  A matrix is
\emph{totally nonnegative} (TN) if every minor is nonnegative and
\emph{totally positive} (TP) if every minor is positive.  A square TN matrix
is \emph{oscillatory} if some positive power of it is TP.  The classical
theory is due to Gantmacher and Krein \cite{GantmacherKrein}; modern accounts
may be found in \cite{FallatJohnson,Pinkus}.

A basic theorem of Gantmacher and Krein \cite{GantmacherKrein} states that an oscillatory matrix
satisfies
\[
A^m\text{ is TP whenever }m\ge n-1.
\]
The least positive integer for which \(A^m\) is TP is called the exponent
of \(A\):
\[
e(A)=\min\{m\ge1:A^m\text{ is TP}\}.
\]
Thus \(e(A)\le n-1\), but its exact value depends on the zero pattern among
the minors of \(A\).

Several authors have studied this exact value.  Fallat gave a useful form of
the oscillatory criterion: a TN matrix is oscillatory if and only if it is
nonsingular and all entries on its first superdiagonal and subdiagonal are
positive \cite[Theorem~1]{FallatRemark}.  Fallat and Liu characterized the
case \(e(A)=n-1\) \cite{FallatLiu}.  Zarai and Margaliot obtained exact
exponents for several classes described by successive elementary
bidiagonal (SEB) factorizations and bounds for other classes
\cite{ZaraiMargaliot}.  They also asked how to determine the exponent
explicitly from the SEB factorization of a general oscillatory matrix
\cite[Section~4]{ZaraiMargaliot}.
The exponent has appeared in dynamics as well: Katz, Margaliot and Fridman
used it to bound subharmonic periods in oscillatory discrete-time systems
\cite{KatzMargaliotFridman}.

There is a related fixed-order theory.  The entries of the \(k\)-th
multiplicative compound are the minors of order \(k\), so eventual positivity
of compounds is the natural language for asking when one prescribed minor
order becomes positive.  Kushel used compounds to characterize eventual
strict total positivity \cite{Kushel}.  Alseidi, Margaliot and Garloff
studied strict sign regularity of order \(k\) and the associated
discrete-time systems
\cite{AlseidiMargaliotGarloff,AlseidiMargaliotGarloffDiscrete}, and Alseidi
and Garloff introduced the exact-order oscillatory exponent used here
\cite{AlseidiGarloff}.  An earlier use of the phrase ``oscillatory of order
\(k\)'', for a cumulative condition on all orders up to \(k\), appears in
\cite{LordPelsser}.

For a fixed oscillatory matrix, minors of different orders need not become
positive at the same power.  Since a matrix is TP exactly when its minors of
every order are positive, the scalar exponent is only the largest of these
first-positive powers; it does not retain the other orderwise thresholds.
Conversely, the prescribed-order theory fixes one order \(k\) and does not
compare different orders for the same matrix.  The vector of least powers at
which the minors of orders \(1,\ldots,n\) become positive is the
\emph{minor-order exponent profile}.  We study how this profile can be
determined from the upper-right and lower-left corner minors of the powers of
the matrix and from an arbitrary adjacent bidiagonal factorization, and how
its coordinates compare.  We also ask whether
restrictions involving only two coordinates account for all compatibility
conditions on the profile.

\section{Preliminaries}
\label{sec:preliminaries}

Put \([n]=\{1,\ldots,n\}\).  For ordered subsets
\(I=\{i_1<\cdots<i_r\}\) and \(J=\{j_1<\cdots<j_r\}\) of equal size, write
\(A[I,J]\) for the corresponding submatrix and
\(\Delta_{I,J}(A)=\det A[I,J]\).  The \(r\)-minors are all such quantities
with \(\lvert I\rvert=\lvert J\rvert=r\).
For \(I\subseteq[n]\), let \(I^c=[n]\setminus I\).  Put
\(L_r=\{1,\ldots,r\}\) and \(R_r=\{n-r+1,\ldots,n\}\).
The minors \(\Delta_{L_r,R_r}(A)\) and \(\Delta_{R_r,L_r}(A)\) are the
remote upper-right and lower-left corner minors of order \(r\).

For a real matrix \(B\), \(B>0\) and \(B\ge0\) mean entrywise strict
positivity and entrywise nonnegativity, respectively.  As in the
Introduction, TN and TP refer to all minors, and a square TN matrix is
oscillatory if one of its positive powers is TP.

Following Alseidi and Garloff \cite{AlseidiGarloff}, we use their
order-\(k\) exponent for oscillatory matrices.

\begin{definition}
Let \(A\) be an \(n\times n\) oscillatory matrix.  Its scalar exponent is
\[
 e(A)=\min\{m\ge1:A^m\text{ is TP}\}.
\]
For \(k=1,\ldots,n\), its \(k\)-th minor-order exponent is
\[
 e_k(A)
 =
 \min\left\{
 m\ge1:
 \text{all \(k\)-minors of }A^m\text{ are positive}
 \right\}.
\]
The vector
\[
 \mathbf e(A)=\bigl(e_1(A),\ldots,e_n(A)\bigr)
\]
is the minor-order exponent profile of \(A\).
\end{definition}
	
\begin{definition}[Multiplicative compound]
	For \(1\le k\le n\), the \(k\)-th multiplicative compound \(C_k(A)\) is
	the \(\binom nk\)-by-\(\binom nk\) matrix whose rows and columns are indexed
	by the \(k\)-subsets of \([n]\) (that is, by the elements of
	\(\binom{[n]}{k}\)), in lexicographic order, and whose
	\((I,J)\)-entry is
	\[
	C_k(A)_{I,J}=\Delta_{I,J}(A).
	\]
	For \(k=1\), the index sets are singletons, so \(C_1(A)=A\).  For
	\(k=n\), the only index set is \([n]\), so \(C_n(A)\) is the
	one-by-one matrix \((\det A)\).
\end{definition}

The Cauchy--Binet formula, applied simultaneously to all \(k\)-minors,
gives
\[
 C_k(AB)=C_k(A)C_k(B),
 \qquad C_k(A^m)=C_k(A)^m.
\]
See \cite[Theorem~1.1.1]{FallatJohnson}.  Since the entries of
\(C_k(A^m)\) are exactly the \(k\)-minors of \(A^m\),
\[
 e_k(A)=\min\{m\ge1:C_k(A)^m>0\}.
\]

\begin{definition}[Adjacent bidiagonal factors and factor sequences]
Let \(E_{ij}\) be the matrix unit with a single \(1\) in position
\((i,j)\).  For \(t\ge0\), set
\[
 x_i(t)=I+tE_{i,i+1},\qquad y_i(t)=I+tE_{i+1,i}
 \quad(1\le i<n).
\]
These are elementary upper and lower bidiagonal factors at position \(i\).
An adjacent bidiagonal factorization is a factorization
\(A=LDU\), where \(D\) is positive diagonal and \(L,U\) are ordered
products of factors \(y_i(t)\) and \(x_i(t)\), respectively.  A factor
with \(t=0\) is the identity matrix and is omitted.  The ordered position
labels of the remaining factors on either side form a one-sided factor
sequence.  A sequence is complete if every position
\(1,\ldots,n-1\) occurs.  The literature also calls this a successive elementary bidiagonal (SEB) factorization \cite{JohnsonOleskyDriessche}.
\end{definition}

Every nonsingular TN matrix has such a factorization; this is the
bidiagonal-factorization form of Neville elimination
\cite{Whitney,Loewner,CryerProperties,GascaPena,
JohnsonOleskyDriessche,FallatJohnson}.  Iterated Cauchy--Binet expresses
each minor of the product as a sum of nonnegative terms.  Consequently a
minor is positive precisely when at least one chain of intermediate index
sets contributes a positive term.
	
We shall use the following classical theorem.

\begin{theorem}[Gantmacher--Krein]\cite{GantmacherKrein}\label{thm:GK}
Let \(A\) be an \(n\times n\) oscillatory matrix.  Then \(A^m\) is TP for
every \(m\ge n-1\).  In particular,
\[
 e(A)\le n-1.
\]
\end{theorem}

The following criterion is due to Fallat \cite[Theorem~1]{FallatRemark}.

\begin{theorem}
\label{thm:oscillatory-criterion}
Let \(A=(a_{ij})\) be an \(n\times n\) TN matrix.  Then \(A\) is
oscillatory if and only if it is nonsingular and
\[
 a_{i,i+1}>0,
 \qquad
 a_{i+1,i}>0,
 \qquad
 i=1,\ldots,n-1.
\]
\end{theorem}

\section{Interval supports of totally nonnegative matrices}
\label{sec:interval-support}

Write a \(k\)-subset as \(I=(i_1<\cdots<i_k)\).  We compare such sets
componentwise:
\[
 I\preceq J
 \quad\Longleftrightarrow\quad
 i_r\le j_r\quad(1\le r\le k),
\]
and, for \(P\preceq Q\), put
\[
 \cint{P}{Q}
 =
 \left\{J\in\binom{[n]}k:P\preceq J\preceq Q\right\}.
\]

For a bond \(a\in\{1,\ldots,n-1\}\), define
\[
 c_a^+(I)=
 \begin{cases}
  (I\setminus\{a\})\cup\{a+1\},&a\in I,\ a+1\notin I,\\
  I,&\text{otherwise},
 \end{cases}
\]
and
\[
 c_a^-(I)=
 \begin{cases}
  (I\setminus\{a+1\})\cup\{a\},&a\notin I,\ a+1\in I,\\
  I,&\text{otherwise}.
 \end{cases}
\]
Thus \(c_a^+\) moves a selected index one place to the right when legal,
whereas \(c_a^-\) is its left-moving counterpart.  Both maps are
order-preserving, and
\[
 c_a^-(I)\preceq I\preceq c_a^+(I).
\]

\begin{lemma}
\label{lem:one-step-interval}
If \(P\preceq Q\), then
\begin{align}
 \cint{P}{Q}\cup c_a^+\bigl(\cint{P}{Q}\bigr)
 &=\cint{P}{c_a^+(Q)},                                      \label{eq:plus-interval}\\
 \cint{P}{Q}\cup c_a^-\bigl(\cint{P}{Q}\bigr)
 &=\cint{c_a^-(P)}{Q}.                                      \label{eq:minus-interval}
\end{align}
\end{lemma}

\begin{proof}
We prove \eqref{eq:plus-interval}.  The left side is contained in the
right side because \(c_a^+\) is order-preserving and inflationary.  If
\(c_a^+(Q)=Q\), this already proves equality.

Suppose \(c_a^+(Q)\ne Q\).  Write
\[
 P=(p_1<\cdots<p_k),\qquad Q=(q_1<\cdots<q_k),
\]
and let \(s\) be the unique coordinate with \(q_s=a\).  Take
\[
 J\in\cint{P}{c_a^+(Q)}\setminus\cint{P}{Q}.
\]
Only the \(s\)-th upper-bound coordinate changed, so \(j_s=a+1\).
Moreover, \(a\notin J\).  Otherwise \(s>1\) and \(j_{s-1}=a\), whereas
\(q_{s-1}\le a-1\), contradicting \(J\preceq c_a^+(Q)\).  Hence
\[
 X=(J\setminus\{a+1\})\cup\{a\}
\]
is a \(k\)-subset and \(c_a^+(X)=J\).  Since \(p_s\le q_s=a\), lowering
the \(s\)-th coordinate preserves \(P\preceq X\), and it restores the old
upper bound \(X\preceq Q\).  Thus \(X\in\cint{P}{Q}\).

For \eqref{eq:minus-interval}, the reverse inclusion again follows from
order preservation.  If \(c_a^-(P)\ne P\), write \(p_s=a+1\).  Any new
point \(J\) has \(j_s=a\) and cannot contain \(a+1\); otherwise
\(j_{s+1}=a+1\), whereas \(p_{s+1}\ge a+2\), contradicting
\(c_a^-(P)\preceq J\).  Replacing \(a\) by \(a+1\) restores the old lower
bound.  Since \(q_s\ge p_s=a+1\), it also preserves the upper bound.
This proves \eqref{eq:minus-interval}.
\end{proof}

A signed bond word is a sequence
\[
 \mathbf q=((\varepsilon_1,q_1),\ldots,(\varepsilon_s,q_s)),
 \qquad \varepsilon_r\in\{+,-\}.
\]
At step \(r\), a trajectory may stay or apply
\(c_{q_r}^{\varepsilon_r}\).  Let \(\mathcal R_{\mathbf q}(I)\) be the
set of endpoints from \(I\).  Define \(G_{\mathbf q}^+(I)\) by taking
every legal \(+\)-move and ignoring every \(-\)-move, and define
\(G_{\mathbf q}^-(I)\) dually.

\begin{theorem}
\label{thm:signed-word-interval}
For every signed bond word and every \(I\in\binom{[n]}k\),
\[
 \mathcal R_{\mathbf q}(I)
 =
 \cint{G_{\mathbf q}^-(I)}{G_{\mathbf q}^+(I)}.
\]
Thus no index set between the two endpoint sets is missing.
\end{theorem}

\begin{proof}
The empty word gives the singleton interval \(\cint{I}{I}\).  If a
prefix has reachable set \(\cint{P}{Q}\), then appending a \(+\)-move
changes it to \(\cint{P}{c_a^+(Q)}\) by
\eqref{eq:plus-interval}, while appending a \(-\)-move changes it to
\(\cint{c_a^-(P)}{Q}\) by \eqref{eq:minus-interval}.  Induction proves
the result.
\end{proof}

The interval identity gives the following support theorem by bidiagonal
factorization and Cauchy--Binet.

\begin{theorem}
\label{thm:interval-support}
Let \(B\) be an \(n\)-by-\(n\) nonsingular TN matrix and
\(1\le k\le n\).  There are intrinsic order-preserving maps
\[
 \lambda_k(B),\rho_k(B):
 \binom{[n]}k\longrightarrow\binom{[n]}k,
 \qquad
 \lambda_k(B)(I)\preceq I\preceq\rho_k(B)(I),
\]
such that
\begin{equation}
\label{eq:interval-support}
 \Delta_{I,J}(B)>0
 \quad\Longleftrightarrow\quad
 \lambda_k(B)(I)\preceq J\preceq\rho_k(B)(I).
\end{equation}
They are the componentwise smallest and largest indices in the corresponding
row of \(C_k(B)\).

If \(A\) and \(B\) are nonsingular TN matrices, then
\begin{equation}
\label{eq:endpoint-composition}
 \lambda_k(AB)=\lambda_k(B)\circ\lambda_k(A),
 \qquad
 \rho_k(AB)=\rho_k(B)\circ\rho_k(A).
\end{equation}
Consequently, for every \(m\ge1\),
\begin{equation}
\label{eq:power-support}
 \Delta_{I,J}(B^m)>0
 \quad\Longleftrightarrow\quad
 \lambda_k(B)^m(I)\preceq J\preceq\rho_k(B)^m(I).
\end{equation}
\end{theorem}

\begin{proof}
By Loewner--Whitney factorization
\cite{Whitney,Loewner,FallatJohnson}, \(B\) is a product of positive
diagonal matrices and factors \(x_a(t),y_a(t)\) with \(t>0\).  Record
\(x_a(t)\) as \((+,a)\), record \(y_a(t)\) as \((-,a)\), and ignore the
positive diagonal factors.

From a row state \(I\), the \(k\)-th compound of \(x_a(t)\) has exactly
the positive transitions \(I\to I\) and \(I\to c_a^+(I)\); for
\(y_a(t)\) the nonstationary transition is \(I\to c_a^-(I)\).
Repeated Cauchy--Binet therefore shows that
\(\Delta_{I,J}(B)>0\) exactly when \(J\) is the endpoint of an optional
trajectory for the recorded signed word.  Theorem
\ref{thm:signed-word-interval} gives \eqref{eq:interval-support}.  The two
endpoints of a finite interval are unique, so the resulting maps do not
depend on the chosen factorization.

Concatenating a factorization of \(A\) with one of \(B\), a lower greedy
trajectory first reaches \(\lambda_k(A)(I)\) and then
\(\lambda_k(B)(\lambda_k(A)(I))\).  The same argument applies to the
upper endpoint, proving \eqref{eq:endpoint-composition}.  Iteration gives
\eqref{eq:power-support}.
\end{proof}

Applying Theorem~\ref{thm:interval-support} to \(B^T\) also shows that every
column support of \(C_k(B)\) is a componentwise interval.

  \begin{remark}
  	For a fixed row index set \(I\), the value \(\lambda_k(B)(I)\) is the
  	leftmost column index set in the positive support of that row, whereas
  	\(\rho_k(B)(I)\) is the rightmost one.  The interval-support formula says
  	more than the existence of these two endpoints: every intermediate index
  	set in componentwise order is also present.  In this sense, a compound row
  	has no holes.
  \end{remark}

\begin{corollary}\label{lem:principal-minors-positive}
Every principal minor of a nonsingular TN matrix is positive.
\end{corollary}

\begin{proof}
At every elementary factor the stationary transition is available.
Equivalently,
\[
\lambda_k(B)(I)\preceq I\preceq\rho_k(B)(I).
\]
Formula \eqref{eq:interval-support} gives \(\Delta_{I,I}(B)>0\).
\end{proof}

\section{Endpoint formulas for exponent profiles}
\label{sec:endpoint-formulas}

Put
\[
 L_k=\{1,\ldots,k\},
 \qquad
 R_k=\{n-k+1,\ldots,n\}.
\]

The all-orders two-corner test for total positivity is classical; see, for
example, \cite{Pinkus}.  The
fixed-order statement below also follows from Pinkus's zero-minor theorem
\cite[Theorem~1.1]{PinkusZero}.  We include a short proof based on
Theorem~\ref{thm:interval-support}.

\begin{theorem}
\label{thm:fixed-order-corner-criterion}
Let \(B\) be an \(n\)-by-\(n\) nonsingular TN matrix and
\(1\le k<n\).  All order-\(k\) minors of \(B\) are positive if and only if
\[
 \Delta_{L_k,R_k}(B)>0,
 \qquad
 \Delta_{R_k,L_k}(B)>0.
\]
\end{theorem}

\begin{proof}
	Only sufficiency needs proof.  Since \(R_k\) is the largest \(k\)-subset,
	the interval-support formula and the first corner condition imply
	\(\rho_k(B)(L_k)=R_k\).  Since \(\rho_k(B)\) is order-preserving and
	\(L_k\preceq I\),
	\[
	R_k=\rho_k(B)(L_k)\preceq\rho_k(B)(I)\preceq R_k,
	\]
	so \(\rho_k(B)(I)=R_k\) for every \(I\).  The second corner condition
	similarly gives \(\lambda_k(B)(I)=L_k\) for every \(I\).
	Formula \eqref{eq:interval-support} now makes every order-\(k\) minor
	positive.
\end{proof}

For an oscillatory matrix \(A\), define the two corner thresholds
\[
 \tau_k^{\mathrm{UR}}(A)
 =
 \min\{m\ge1:\Delta_{L_k,R_k}(A^m)>0\},
\]
\[
 \tau_k^{\mathrm{LL}}(A)
 =
 \min\{m\ge1:\Delta_{R_k,L_k}(A^m)>0\},
 \qquad
 \tau_k(A)=\max\{\tau_k^{\mathrm{UR}}(A),
                       \tau_k^{\mathrm{LL}}(A)\}.
\]
They are finite because a sufficiently high power of \(A\) is TP.

\begin{theorem}
\label{thm:exact-corner-formula}
Let \(A\) be oscillatory, let \(1\le k<n\), and abbreviate
\(\lambda_k=\lambda_k(A)\), \(\rho_k=\rho_k(A)\).  Then
\[
 e_k(A)
 =
 \tau_k(A)
 =
 \max\left\{
   \min\{m\ge1:\rho_k^m(L_k)=R_k\},
   \min\{m\ge1:\lambda_k^m(R_k)=L_k\}
 \right\}.
\]
More generally, the least nonnegative integer \(m\) for which the
\((I,J)\)-minor of \(A^m\) is positive is
\begin{equation}
\label{eq:all-pairs-hitting}
 \min\{m\ge0:\lambda_k^m(I)\preceq J\preceq\rho_k^m(I)\}.
\end{equation}
The largest of these all-pairs hitting times is attained by one of
\((L_k,R_k)\) and \((R_k,L_k)\), and equals \(e_k(A)\).
\end{theorem}

\begin{proof}
Formula \eqref{eq:all-pairs-hitting} is
\eqref{eq:power-support}.  Applying
Theorem~\ref{thm:fixed-order-corner-criterion} to \(A^m\) gives the
formula for \(e_k(A)\).

Let \(M\) be the larger extreme hitting time.  Then
\[
 \rho_k^M(L_k)=R_k,
 \qquad
 \lambda_k^M(R_k)=L_k.
\]
For every \(I\), monotonicity gives
\(\rho_k^M(I)=R_k\) and \(\lambda_k^M(I)=L_k\).  Hence every ordered pair
is reached by time \(M\), while the two extreme pairs show that no smaller
uniform bound works.
\end{proof}

\begin{lemma}\label{lem:fixed-order-persistence}
Let \(B\) be nonsingular and TN.  If
\(\Delta_{I,J}(B^m)>0\), then \(\Delta_{I,J}(B^s)>0\) for every
\(s\ge m\).  In particular, if all order-\(k\) minors of \(B^m\) are
positive, the same is true at every later power.
\end{lemma}

\begin{proof}
Since \(\lambda_k(B)(X)\preceq X\preceq\rho_k(B)(X)\) and both endpoint maps
are order-preserving,
\[
 \lambda_k(B)^{s}(I)\preceq\lambda_k(B)^m(I),
 \qquad
 \rho_k(B)^m(I)\preceq\rho_k(B)^{s}(I)
\]
whenever \(s\ge m\).  The claim follows from
\eqref{eq:power-support}.
\end{proof}

For an oscillatory matrix, positivity at every fixed order therefore
persists, and
\[
 e(A)=\max_{1\le k\le n}e_k(A).
\]

\begin{corollary}
\label{thm:power-profile-scaling}
Let \(A\) be oscillatory and let \(q\) be a positive integer.  Then
\(A^q\) is oscillatory and
\[
 e_k(A^q)=\left\lceil\frac{e_k(A)}q\right\rceil,
 \qquad
 e(A^q)=\left\lceil\frac{e(A)}q\right\rceil.
\]
\end{corollary}

\begin{proof}
The \(k\)-minors of \((A^q)^m=A^{qm}\) are all positive exactly when
\(qm\ge e_k(A)\), by Lemma~\ref{lem:fixed-order-persistence}.  This gives
the first formula; taking the maximum over \(k\) gives the second.
\end{proof}

\section{Sharp bounds and comparison between orders}
\label{sec:bounds-comparison}

Let \(Q=(q_1,\ldots,q_\ell)\) be a sequence of adjacent positions in
\(1,\ldots,n-1\).  A \(Q\)-pass applies
\(c_{q_1}^+,\ldots,c_{q_\ell}^+\) in that order, and \(Q\) is
\emph{complete} if every position occurs.  Let \(t_k(Q)\) be the number of
repeated passes needed to carry \(L_k\) to \(R_k\), with value \(\infty\) if
this never occurs.

\begin{lemma}
\label{lem:greedy-reachability}
Let \(B=LDU\) be an adjacent bidiagonal factorization of a nonsingular TN
matrix.  Delete zero-parameter factors.  Record the upper position sequence
of \(U\) as \(Q_+\), and record the upper position sequence of \(L^T\), in
its factor order, as \(Q_-\).  Then, for every \(m\ge1\),
\begin{align}
 \Delta_{L_k,R_k}(B^m)>0
 &\quad\Longleftrightarrow\quad
 c_{Q_+^m}^+(L_k)=R_k,                                    \label{eq:one-sided-upper-sequence}\\
 \Delta_{R_k,L_k}(B^m)>0
 &\quad\Longleftrightarrow\quad
 c_{Q_-^m}^+(L_k)=R_k,                                    \label{eq:one-sided-lower-sequence}
\end{align}
where \(c_Q^+\) denotes the successive application of the maps indexed by
\(Q\).  If \(B=A\) is oscillatory, both sequences are complete and
\begin{equation}
\label{eq:word-profile}
 \tau_k^{\mathrm{UR}}(A)=t_k(Q_+),
 \qquad
 \tau_k^{\mathrm{LL}}(A)=t_k(Q_-),
 \qquad
 e_k(A)=\max\{t_k(Q_+),t_k(Q_-)\}.
\end{equation}
\end{lemma}

\begin{proof}
In the factorization sequence for \(B\), the upper endpoint map ignores every
lower and diagonal factor and takes every legal upper move.  Hence
\(\rho_k(B)^m(L_k)\) is the result of \(m\) greedy \(Q_+\)-passes.
Transposition gives the corresponding statement for the lower corner and
\(Q_-\).  The power-support formula proves
\eqref{eq:one-sided-upper-sequence}--\eqref{eq:one-sided-lower-sequence}.
When \(B=A\) is oscillatory, Theorem~\ref{thm:exact-corner-formula} proves
\eqref{eq:word-profile}.

It remains to check completeness.  If position \(i\) did not occur with a
positive parameter in \(U\), then \(U\) would have a zero block across
the cut \([i]\mid\{i+1,\ldots,n\}\), and lower triangularity of \(L\)
would give \(a_{i,i+1}=0\).  For an oscillatory matrix this contradicts
Theorem~\ref{thm:oscillatory-criterion}.  Applying the same argument to
\(A^T\) proves completeness of \(Q_-\).
\end{proof}

The comparisons below use one rectangular bookkeeping device.  Put
\[
 \mathcal R_k=[k]\times[n-k],
 \qquad
 \lambda_k(p,b)=k-p+b.
\]
The cell \((p,b)\) records the unique adjacent exchange between the \(p\)-th
initially selected index, counted from the right, and the \(b\)-th initially
unselected index, counted from the left.  Its bond is
\(\lambda_k(p,b)\).  It can occur only after the cells \((p-1,b)\) and
\((p,b-1)\), when these cells exist.  Thus the whole motion from \(L_k\) to
\(R_k\) is represented by the ordinary product order on the rectangle
\(\mathcal R_k\).

\begin{lemma}
\label{lem:labelled-rectangle-recurrence}
Let \(Q=(q_1,\ldots,q_\ell)\) be complete and repeat it periodically, so
\(q_{r+\ell}=q_r\).  For a bond \(d\), put
\[
 N_d(s)=\min\{r>s:q_r=d\}.
\]
If \(T^{(k)}_{p,b}\) is the occurrence number at which cell \((p,b)\) of
\(\mathcal R_k\) is completed, then
\begin{equation}
\label{eq:crossing-recurrence}
 T^{(k)}_{p,b}
 =N_{\lambda_k(p,b)}
 \left(\max\{T^{(k)}_{p-1,b},T^{(k)}_{p,b-1}\}\right),
 \qquad
 T^{(k)}_{0,b}=T^{(k)}_{p,0}=0.
\end{equation}
Consequently
\[
 t_k(Q)=
 \left\lceil\frac{T^{(k)}_{k,n-k}}\ell\right\rceil.
\]
More generally, suppose a number \(\widehat T_{p,b}\) is assigned to every
cell so that it is an occurrence of the correct bond and is strictly larger
than the numbers assigned to the existing predecessor cells.  Then
\(T^{(k)}_{p,b}\le\widehat T_{p,b}\) for every cell.
\end{lemma}

\begin{proof}
The two predecessors express the only two possible obstructions to the
exchange represented by \((p,b)\).  The first subsequent occurrence of its
bond therefore gives \eqref{eq:crossing-recurrence}.  The southeast corner
is the last cell in the product order, proving the formula for \(t_k(Q)\).
The final assertion follows by induction on \(p+b\): the recurrence always
chooses the earliest admissible occurrence.
\end{proof}

\begin{lemma}
\label{lem:comparator-pass}
Let \(h,k\ge1\), and let \(c\) be any ordering of the
\(k+h-1\) adjacent positions.  Repeating this order carries \(L_k\) to
\(R_k\) in at most
\[
 \max\{k,h\}
\]
passes.
\end{lemma}

\begin{proof}
Let \(\pi(d)\) be the position of bond \(d\) in the pass and let \(s_{p,b}\)
be the pass in which cell \((p,b)\) of \(\mathcal R_k\) is completed.  By
Lemma~\ref{lem:labelled-rectangle-recurrence}, \(s_{p,b}\) is one plus the
largest number of decreases of \(\pi\) along a lattice path from \((1,1)\)
to \((p,b)\), where the cell \((u,v)\) has label \(k-u+v\).

Such a path has \(b-1\) steps that raise the label and \(p-1\) steps that
lower it.  Pair opposite traversals of each edge of the label line.  In a
paired traversal exactly one direction can decrease \(\pi\).  If
\(b-1\ge p-1\), the net traversal of every edge is either zero or one step
in the label-increasing direction, so the number of decreases is at most
the \(b-1\) increasing steps.  Reflection gives the bound \(p-1\) when
\(p-1\ge b-1\).  Therefore
\[
 s_{p,b}\le\max\{p,b\},
\]
and the southeast cell is completed by pass \(\max\{k,h\}\).
\end{proof}

\begin{theorem}
\label{thm:universal-fixed-order-bound}
Every \(n\)-by-\(n\) oscillatory matrix satisfies
\[
 e_k(A)\le\max\{k,n-k\},
 \qquad 1\le k<n.
\]
The bound is sharp for every pair \((n,k)\).
\end{theorem}

\begin{proof}
Take the complete sequences \(Q_+,Q_-\) in
Lemma~\ref{lem:greedy-reachability}.  Select the first occurrence of each
position in either sequence, retaining their order.  Additional legal
adjacent moves cannot
delay a greedy trajectory, so Lemma~\ref{lem:comparator-pass} gives
\[
 t_k(Q_+),t_k(Q_-)\le\max\{k,n-k\}.
\]
Equation \eqref{eq:word-profile} proves the upper bound.

For sharpness, put
\[
 c^\uparrow=(1,2,\ldots,n-1),
 \qquad
 c^\downarrow=(n-1,n-2,\ldots,1),
\]
and, for a position sequence \(c=(c_1,\ldots,c_s)\), set
\[
 U(c)=\prod_{r=1}^s(I+E_{c_r,c_r+1}),
 \qquad A_c=U(c)^TU(c).
\]
Every position occurs, so \(A_c\) is nonsingular TN with positive first
off-diagonal entries and is oscillatory.  Both one-sided sequences in its
adjacent bidiagonal factorization are \(c\); hence
\[
 e_j(A_c)=t_j(c).
\]
In an increasing pass every cell in row \(p\) of \(\mathcal R_j\)
is completed in pass \(p\), so \(t_j(c^\uparrow)=j\).  Reflection gives
\(t_j(c^\downarrow)=n-j\).  Choosing \(c^\uparrow\) when
\(k\ge n-k\), and \(c^\downarrow\) otherwise, yields equality.
\end{proof}

\subsection{Neighboring minor orders}

Consecutive components among \(e_1(A),\ldots,e_{n-1}(A)\) differ by at most one.
The proof compares the same
periodic sequence of adjacent moves at two neighboring minor orders.

Let \(w=(c_1,\ldots,c_\ell)\) be a sequence in
\(1,\ldots,n-1\), with every position occurring at least once.  Repeat \(w\)
periodically.  Thus \(t_k(w)\) is the number of passes needed to carry
\(L_k\) to \(R_k\).

\begin{lemma}
	\label{lem:one-period-interlacing}
	For every \(k=1,\ldots,n-2\),
	\[
	|t_{k+1}(w)-t_k(w)|\le1.
	\]
\end{lemma}

\begin{proof}
	Put \(h=n-k\), let \(T_{p,b}=T^{(k)}_{p,b}\), and write
	\(m=t_k(w)\).  We construct an admissible filling of the
	\((k+1)\)-by-\((h-1)\) rectangle from the \(k\)-by-\(h\) filling:
	\[
	 \widehat S_{1,b}=T_{1,b+1},
	 \qquad
	 \widehat S_{p,b}=T_{p-1,b}+\ell
	 \quad(2\le p\le k+1).
	\]
	The bond labels agree, because
	\[
	 \lambda_{k+1}(1,b)=\lambda_k(1,b+1),
	 \qquad
	 \lambda_{k+1}(p,b)=\lambda_k(p-1,b).
	\]
	All predecessor inequalities inside the first row and inside the lower
	block are inherited from \(\mathcal R_k\).  Along their common boundary,
	\[
	 T_{1,b+1}<T_{1,b}+\ell.
	\]
	Indeed, \(T_{1,b+1}\) is the first occurrence of the next bond after
	\(T_{1,b}\); completeness supplies such an occurrence within one period,
	and equality is impossible because the two bonds are distinct.
	Thus the displayed assignment is admissible in the sense of
	Lemma~\ref{lem:labelled-rectangle-recurrence}.  Its largest entry is at
	most \((m+1)\ell\), so
	\begin{equation}
	\label{eq:one-sided-neighbor}
	 t_{k+1}(w)\le t_k(w)+1.
	\end{equation}

	Reflecting the positions changes \(w\) into
	\[
	 w^*=(n-c_1,\ldots,n-c_\ell)
	\]
	and gives \(t_k(w)=t_{n-k}(w^*)\).  Applying
	\eqref{eq:one-sided-neighbor} to \(w^*\) at order \(n-k-1\) yields the
	reverse inequality.
\end{proof}

\begin{theorem}
	\label{thm:neighboring-orders}
	Every \(n\times n\) oscillatory matrix satisfies
	\[
	|e_{k+1}(A)-e_k(A)|\le1,
	\qquad 1\le k\le n-2.
	\]
\end{theorem}

\begin{proof}
	Write an adjacent bidiagonal factorization \(A=LDU\).  Delete the zero-multiplier upper
	factors and record the bond labels of the remaining factors, in their
	factor order, as a word \(w_U\).  The proof of
	Theorem~\ref{thm:universal-fixed-order-bound} shows that every bond occurs
	in \(w_U\).  Define \(w_{L^T}\) analogously from \(L^T\).

	Lemma~\ref{lem:greedy-reachability} gives
	\[
	\tau_k^{\mathrm{UR}}(A)=t_k(w_U),
	\qquad
	\tau_k^{\mathrm{LL}}(A)=t_k(w_{L^T}).
	\]
	Adjacent terms of both sequences differ by at most one, by
	Lemma~\ref{lem:one-period-interlacing}.  Theorem~\ref{thm:exact-corner-formula}
	gives
	\[
	e_k(A)=\max\{t_k(w_U),t_k(w_{L^T})\}.
	\]
	If adjacent terms of two sequences differ by at most one, the same is true
	of their pointwise maximum.  This proves the result.
\end{proof}

The same labelled rectangle compares any two orders directly.  Unlike
repeated use of the neighboring-order inequality, the resulting loss is
controlled by a single ratio of row counts or column counts.

\begin{theorem}
	\label{thm:all-order-comparison}
	Let \(A\) be an \(n\times n\) oscillatory matrix.  For every
	\(1\le j,k\le n-1\),
	\[
	e_j(A)\le Q_{j,k}e_k(A),
	\]
	where
	\[
	Q_{j,k}
	=
	\left\lceil
	\max\left\{\frac{j}{k},\frac{n-j}{n-k}\right\}
	\right\rceil
	=
	\begin{cases}
	\displaystyle
	\left\lceil\dfrac{n-j}{n-k}\right\rceil,&j\le k,\\[6pt]
	\displaystyle
	\left\lceil\dfrac{j}{k}\right\rceil,&j\ge k.
	\end{cases}
	\]
	More generally, every complete adjacent factor sequence \(w\) satisfies
	\[
		t_j(w)\le Q_{j,k}t_k(w)
		\qquad(1\le j,k\le n-1).
	\]
\end{theorem}

\begin{proof}
	We first prove the stated one-sided assertion.  Let
	\(w=(c_1,\ldots,c_\ell)\) be a complete word on the bonds
	\(1,\ldots,n-1\), and put
	\[
	m=t_k(w),
	\qquad
	h=n-k,
	\]
	and let \(T_{p,b}=T^{(k)}_{p,b}\) be the cell-completion numbers in
	\(\mathcal R_k\).  Every cell precedes the southeast cell, and the
	rectangle is completed in \(m\) passes.  Hence
	\begin{equation}
	\label{eq:crossing-window}
	T_{p,b}\le m\ell
	\qquad
	(1\le p\le k,\ 1\le b\le h).
	\end{equation}
	We shall partition \(\mathcal R_j\) into slabs and embed each slab into
	\(\mathcal R_k\), preserving both bond labels and predecessor relations.
	After successive period shifts, the resulting cell numbers are admissible
	in the sense of Lemma~\ref{lem:labelled-rectangle-recurrence}.

	Suppose first that \(j\le k\).  Put
	\[
	H=n-j,
	\qquad
	\delta=k-j=H-h,
	\qquad
	q=\left\lceil\frac{H}{h}\right\rceil.
	\]
	Partition the column indices \(1,\ldots,H\) into \(q\) consecutive blocks
	\[
	B_s=\{a_s+1,\ldots,a_s+d_s\},
	\qquad d_s\le h.
	\]
	For each block set
	\[
	r_s=\max\{\delta-a_s,0\},
	\qquad
	u_s=\max\{a_s-\delta,0\}.
	\]
	Then
	\[
	0\le r_s\le k-j,
	\qquad
	0\le u_s\le h-d_s,
	\qquad
	r_s-u_s=\delta-a_s.
	\]
	Thus the map
	\[
	(p,a_s+b)\longmapsto(p+r_s,b+u_s),
	\qquad 1\le p\le j,\quad1\le b\le d_s,
	\]
	embeds this slab into \(\mathcal R_k\) and preserves its bond labels,
	because
	\[
	k-(p+r_s)+(b+u_s)=j-p+(a_s+b).
	\]
	Assign to the target cell \((p,a_s+b)\) the occurrence number
	\[
	\widehat T_{p,a_s+b}
	=(s-1)m\ell+T_{p+r_s,b+u_s}.
	\]
	Within a block these assigned numbers increase along both predecessor
	relations.  At a block boundary, every
	assigned number in the preceding block is at most \((s-1)m\ell\), whereas
	every assigned number in the new block is larger than \((s-1)m\ell\).
	Periodicity preserves the required bond label.  Hence the assigned numbers
	form an admissible filling.  The labelled-rectangle recurrence shows that
	the actual \(j\)-set motion finishes no later, and therefore
	\begin{equation}
	\label{eq:one-sided-all-order-left}
	t_j(w)
	\le
	\left\lceil\frac{n-j}{n-k}\right\rceil t_k(w)
	\qquad(j\le k).
	\end{equation}

	Now suppose that \(j\ge k\).  Put
	\[
	H=n-j,
	\qquad
	\delta=j-k=h-H,
	\qquad
	q=\left\lceil\frac{j}{k}\right\rceil,
	\]
	and partition the row indices into consecutive blocks
	\[
	P_s=\{a_s+1,\ldots,a_s+d_s\},
	\qquad d_s\le k.
	\]
	Set
	\[
	r_s=\max\{a_s-\delta,0\},
	\qquad
	u_s=\max\{\delta-a_s,0\}.
	\]
	Now
	\[
	0\le r_s\le k-d_s,
	\qquad
	0\le u_s\le h-H,
	\qquad
	u_s-r_s=\delta-a_s.
	\]
	The map
	\[
	(a_s+p,b)\longmapsto(p+r_s,b+u_s),
	\qquad1\le p\le d_s,\quad1\le b\le H,
	\]
	again preserves bond labels:
	\[
	k-(p+r_s)+(b+u_s)=j-(a_s+p)+b.
	\]
	Assign the corresponding cell number, shifted by \((s-1)m\ell\), to
	each target cell in \(P_s\).  The same argument, now with row-slab
	boundaries in place of column-slab boundaries, yields
	\begin{equation}
	\label{eq:one-sided-all-order-right}
	t_j(w)
	\le
	\left\lceil\frac{j}{k}\right\rceil t_k(w)
	\qquad(j\ge k).
	\end{equation}

	Finally apply \eqref{eq:one-sided-all-order-left} and
	\eqref{eq:one-sided-all-order-right} to the positive-factor words \(w_U\)
	and \(w_{L^T}\) in an adjacent bidiagonal factorization \(A=LDU\).
	Lemma~\ref{lem:greedy-reachability} and
	Theorem~\ref{thm:exact-corner-formula} give
	\[
	e_r(A)=\max\{t_r(w_U),t_r(w_{L^T})\}.
	\]
	Taking coordinatewise maxima preserves the two inequalities and proves the
	result.
\end{proof}

\begin{remark}
	Taking \(j=1\) and \(j=n-1\) gives
	\[
	e_1(A)\le
	\left\lceil\frac{n-1}{n-k}\right\rceil e_k(A),
	\qquad
	e_{n-1}(A)\le
	\left\lceil\frac{n-1}{k}\right\rceil e_k(A).
	\]
	In particular,
	\[
	e_1(A)\le2e_k(A)
	\quad\left(2\le k\le\frac{n+1}{2}\right),
	\qquad
	e_{n-1}(A)\le2e_k(A)
	\quad\left(\frac{n-1}{2}\le k\le n-2\right).
	\]
	The symmetric family in Theorem~\ref{thm:all-order-sharpness} below
	attains both branches simultaneously, and therefore also attains the
	factor two whenever the displayed ceiling is two.

	Fallat and Liu obtained related exact identities among one-sided step
	counts for their class of basic oscillatory matrices
	\cite[Lemmas~13 and~15]{FallatLiu}.  The block embedding above applies to
	every complete factor word and hence to arbitrary oscillatory matrices.
\end{remark}

For each reference order, one symmetric matrix attains all the constants in
Theorem~\ref{thm:all-order-comparison} simultaneously.  For a word
\(w=(i_1,\ldots,i_s)\) on the bonds \(1,\ldots,n-1\), write
\[
U(w)=\prod_{r=1}^s(I+E_{i_r,i_r+1}).
\]

\begin{theorem}
	\label{thm:all-order-sharpness}
	Fix \(1\le k\le n-1\), and let
	\[
	\omega_k
	=
	(k,k-1,\ldots,1)
	(k+1,k,\ldots,2)\cdots
	(n-1,n-2,\ldots,n-k).
	\]
	Then the symmetric oscillatory matrix
	\[
	A_k=U(\omega_k)^TU(\omega_k)
	\]
	satisfies \(e_k(A_k)=1\) and, simultaneously for every
	\(1\le j\le n-1\),
	\[
	e_j(A_k)
	=
	Q_{j,k}
	=
	\left\lceil
	\max\left\{\frac{j}{k},\frac{n-j}{n-k}\right\}
	\right\rceil.
	\]
	Thus every constant in Theorem~\ref{thm:all-order-comparison} is optimal,
	already within the symmetric oscillatory class.
\end{theorem}

\begin{proof}
	Every bond occurs in \(\omega_k\).  Hence \(U(\omega_k)\) is nonsingular
	TN with a positive first superdiagonal, and
	\(A_k=U(\omega_k)^TU(\omega_k)\) is nonsingular TN with both first
	off-diagonals positive.  Theorem~\ref{thm:oscillatory-criterion}
	therefore shows that \(A_k\) is oscillatory.

	One \(\omega_k\)-pass carries \(L_k\) to \(R_k\): in block \(b\), the labels
	\(k+b-1,\ldots,b\) complete, in order, the cells
	\((1,b),\ldots,(k,b)\).  Thus the successive blocks enumerate the entire
	rectangle \(\mathcal R_k\).  Both one-sided factor words of \(A_k\) are
	\(\omega_k\), so Lemma~\ref{lem:greedy-reachability} gives
	\[
	e_j(A_k)=t_j(\omega_k).
	\]
	In particular, \(e_k(A_k)=1\).
	Theorem~\ref{thm:all-order-comparison}, together with
	\(e_k(A_k)=1\), gives \(e_j(A_k)\le Q_{j,k}\).

	It remains to prove the reverse inequality directly from the labelled
	rectangles.
	Put \(h=n-k\).  An increasing consecutive-label subsequence of
	\(\omega_k\) uses at most one label from each of its \(h\) decreasing
	blocks, while \(1,2,\ldots,h\), chosen from successive blocks, attains
	length \(h\).  Along a fixed row of \(\mathcal R_j\), cell labels are
	successive increasing bonds, so at most \(h\) cells in that row can be
	completed in one pass.  Since the row has \(n-j\) cells,
	\[
	t_j(\omega_k)\ge
	\left\lceil\frac{n-j}{n-k}\right\rceil.
	\]

	The longest decreasing consecutive-label subsequence of \(\omega_k\)
	has length \(k\).  To see this, reverse the word.  Its blocks become the
	increasing intervals
	\[
	(h,\ldots,h+k-1),(h-1,\ldots,h+k-2),\ldots,(1,\ldots,k).
	\]
	One block gives length \(k\).  A consecutive increasing subsequence of
	length \(k+1\) would have to end in a block whose initial label is larger
	than that of the block containing its first label, contrary to the
	decreasing order of the block initials.  Thus at most \(k\) cells in a
	fixed column of \(\mathcal R_j\) can be completed in one pass.  Since the
	column has \(j\) cells,
	\[
	t_j(\omega_k)\ge\left\lceil\frac{j}{k}\right\rceil.
	\]
	The two lower bounds combine to \(e_j(A_k)\ge Q_{j,k}\), completing the
	proof.
\end{proof}

\section{Checkerboard inversion}
\label{sec:checkerboard-duality}

If \(M\) is nonsingular, \(I,J\) have the same cardinality, and
\(\sigma(I)=\sum_{i\in I}i\), Jacobi's identity states that
\begin{equation}
\label{eq:jacobi-complementary-minor}
\det M^{-1}[I,J]
=
(-1)^{\sigma(I)+\sigma(J)}
\frac{\det M[J^c,I^c]}{\det M}.
\end{equation}
See, for example, \cite[Section~1.2]{FallatJohnson}.
Put
\[
 D=\operatorname{diag}(1,-1,1,-1,\ldots),
 \qquad D_{ii}=(-1)^{i-1}.
\]

\begin{definition}\label{def:checkerboard-inverse}
For a nonsingular \(n\times n\) matrix \(A\), its checkerboard inverse is
\[
 A^\vee=DA^{-1}D.
\]
\end{definition}

\begin{theorem}
\label{thm:checkerboard-inverse-duality}
Let \(A\) be an \(n\times n\) nonsingular TN matrix.  Then \(A^\vee\) is
nonsingular TN and
\[
 (A^\vee)^\vee=A.
\]
Moreover, \(A\) is oscillatory if and only if \(A^\vee\) is oscillatory.
If these conditions hold, then
\[
 e_k(A^\vee)=e_{n-k}(A)
 \quad(1\le k<n),
 \qquad e_n(A^\vee)=e_n(A)=1,
\]
and consequently \(e(A^\vee)=e(A)\).
\end{theorem}

\begin{proof}
The signs contributed by the two factors \(D\) cancel the sign in
\eqref{eq:jacobi-complementary-minor}.  Hence, for equal-sized \(I,J\),
\[
 \det A^\vee[I,J]
 =\frac{\det A[J^c,I^c]}{\det A}.
\]
Since \(A\) is nonsingular TN, the right side is nonnegative and
\(\det A>0\).  Thus \(A^\vee\) is TN and nonsingular.  Also
\[
 (A^\vee)^{-1}=DAD,
 \qquad (A^\vee)^\vee=D(A^\vee)^{-1}D=A.
\]

For every \(m\ge1\), \((A^\vee)^m=DA^{-m}D\).  Applying the preceding
minor formula to \(A^m\) gives
\[
 \det\bigl((A^\vee)^m[I,J]\bigr)
 =
 \frac{\det\bigl((A^m)[J^c,I^c]\bigr)}{\det(A^m)}.
\]
Thus \((A^\vee)^m\) is TP if and only if \(A^m\) is TP, and all its
\(k\)-minors are positive if and only if all \((n-k)\)-minors of \(A^m\)
are positive.  This proves the assertions about oscillation and exponents.
The determinant exponents equal one because both matrices are nonsingular.
\end{proof}

The corner thresholds also reverse.  Indeed,
\(R_r^c=L_{n-r}\), \(L_r^c=R_{n-r}\), and Jacobi's identity give
\[
\det\bigl((A^\vee)^m[L_r,R_r]\bigr)
=
\frac{\det\bigl((A^m)[L_{n-r},R_{n-r}]\bigr)}{\det(A^m)},
\]
and the analogous identity with the two corners reversed.  Since
\(\det(A^m)>0\),
\[
\tau_r^{\mathrm{UR}}(A^\vee)=\tau_{n-r}^{\mathrm{UR}}(A),
\qquad
\tau_r^{\mathrm{LL}}(A^\vee)=\tau_{n-r}^{\mathrm{LL}}(A),
\]
so \(\tau_r(A^\vee)=\tau_{n-r}(A)\).  Moreover,
\[
(A^q)^\vee=(A^\vee)^q,
\]
and the power-scaling and profile-reversal formulas give
\[
e_k\bigl((A^q)^\vee\bigr)
=
\left\lceil\frac{e_{n-k}(A)}q\right\rceil.
\]
In particular, \(A^\vee=A\) forces \(e_k(A)=e_{n-k}(A)\).

\section{Adjacent factor patterns and permutation ranks}
\label{sec:factor-patterns}

The preceding results determine the profile from two adjacent factor
sequences.  Each sequence can be replaced by one permutation without
changing its action on any index set.  The ranks of southwest rectangles in the
permutation matrix will then determine every exponent threshold.

Associate the following permutation with one pass.

\begin{definition}[Permutation attached to a factor sequence]
Let \(Q=(q_1,\ldots,q_s)\).  Begin with the ordered list
\(z_Q^{(0)}=(1,2,\ldots,n)\).  At position \(q_r\), interchange the two
entries if the left entry is smaller.  After one pass write
\(\sigma_Q=z_Q^{(1)}\); after \(m\) passes write \(z_Q^{(m)}\).
\end{definition}

\begin{example}
For \(n=4\) and \(Q=(2,1,3,2,2)\), the labelled list evolves as
\[
 1234\longmapsto1324\longmapsto3124\longmapsto3142
 \longmapsto3412\longmapsto3412.
\]
Thus \(\sigma_Q=(3,4,1,2)\).  For \(k=2\), the labels \(1,2\) finish in
the last two positions after one pass, so \(t_2(Q)=1\).  The final
ineffective comparison illustrates why a long factor sequence can be
recorded by a shorter permutation.
\end{example}

\begin{theorem}
\label{thm:label-sorting}
For every \(1\le k<n\),
\begin{equation}
\label{eq:threshold-label-set}
 c_{Q^m}^+(L_k)=\{r:(z_Q^{(m)})_r\le k\}.
\end{equation}
Consequently,
\[
 t_k(Q)=
 \min\left\{m\ge1:
 \{(z_Q^{(m)})_{n-k+1},\ldots,(z_Q^{(m)})_n\}=[k]
 \right\}.
\]
If \(Q\) is complete, every \(t_k(Q)\) is at most \(n-1\).
\end{theorem}

\begin{proof}
For a fixed \(k\), retain only the positions occupied by labels in \([k]\).
At an adjacent exchange this set either stays fixed or changes by exactly
the map \(c_i^+\).  Induction over the factors and the repetitions proves
\eqref{eq:threshold-label-set}.  The second assertion says precisely that
these positions have become \(R_k\); the final bound is
Lemma~\ref{lem:comparator-pass}, after retaining the first occurrence of
each adjacent position.  The omitted moves can only advance the selected
positions and therefore cannot delay the target.
\end{proof}

\begin{lemma}
\label{lem:three-adjacent-identities}
Both the adjacent exchanges on labelled lists and the maps \(c_i^+\) on
index sets satisfy
\[
 c_i^+c_i^+=c_i^+,
 \qquad
 c_i^+c_j^+=c_j^+c_i^+\quad(|i-j|>1),
 \qquad
 c_i^+c_{i+1}^+c_i^+
 =c_{i+1}^+c_i^+c_{i+1}^+.
\]
Consequently, the action of \(Q\), on every labelled list and every index
set, is determined by the single permutation \(\sigma_Q\).
\end{lemma}

\begin{proof}
The first identity says that a second attempt at the same right shift does
nothing.  The second concerns disjoint pairs of coordinates.  For the third,
only membership of \(i,i+1,i+2\) matters; checking its eight possible
subsets gives the same final subset on both sides.  On a labelled list the
third identity simply sorts three entries into decreasing order.

It is standard that two shortest adjacent-exchange sequences for the same
permutation are connected by the disjoint-pair and three-position identities
\cite[Theorem~3.3.1]{BjornerBrenti}.  We include an elementary induction.  In
a shortest sequence every inverted pair of labels
is exchanged exactly once.  The largest label can only move left.  Starting
with its last exchange and working backward, commute past exchanges on disjoint
pairs; the only three-label obstruction is moved past by the three-position
identity.  In this way the exchanges involving the largest label form the
same suffix in both sequences.  Delete that suffix and apply the
same argument to the other \(n-1\) labels.

Proceed through an arbitrary sequence \(Q\) from left to right.  Inductively
choose a minimum sequence \(R\) for the labelled permutation obtained so far,
with the same index map.  If the next exchange at position \(i\) is
effective, then \(Ri\) is still minimum.  If it is ineffective, the current
permutation has a descent at \(i\).  Swapping that descent first gives a
permutation with one fewer inversion, so a minimum sequence for the current
permutation may be chosen in the form \(R'i\).  By the preceding paragraph
we may replace \(R\) by \(R'i\); after appending the ineffective exchange we
obtain \(R'ii\), whose last repetition is ineffective.  Thus every step
preserves both actions, and the final actions depend only on \(\sigma_Q\).
\end{proof}

For a permutation \(\pi\), let \(S_\pi\) be the common adjacent map, on a
labelled list or an index set, given by any shortest adjacent sequence
producing \(\pi\).  Put
\[
 z_\pi^{(0)}=(1,\ldots,n),
 \qquad
 z_\pi^{(m)}=S_\pi^m(z_\pi^{(0)}).
\]
If \(\pi=\sigma_Q\), then \(z_\pi^{(m)}=z_Q^{(m)}\).  We also write
\[
 t_k(\pi)=
 \min\left\{m\ge1:
 \{(z_\pi^{(m)})_{n-k+1},\ldots,(z_\pi^{(m)})_n\}=[k]\right\},
\]
so \(t_k(Q)=t_k(\sigma_Q)\).

For any permutation \(u=(u_1,\ldots,u_n)\), define its southwest rectangle
counts by
\[
 d_u(p,q)=\#\{i>p:u_i\le q\},
 \qquad0\le p,q\le n.
\]
If \(P_Q\) is the permutation matrix of \(\sigma_Q\), then
\begin{equation}
\label{eq:rank-rectangle}
d_Q(p,q):=d_{\sigma_Q}(p,q)
 =\operatorname{rank}P_Q[\{p+1,\ldots,n\},\{1,\ldots,q\}]
 =\#\{i>p:\sigma_i\le q\}.
\end{equation}
For later repetitions set
\[
 d_\pi^{(m)}(p,q)=d_{z_\pi^{(m)}}(p,q),
 \qquad d_\pi^{(1)}=d_\pi.
\]

\begin{lemma}
\label{lem:rectangle-rank-product}
For permutations \(u,\pi\),
\begin{equation}
\label{eq:rectangle-rank-product}
 d_{S_\pi(u)}(p,q)
 =\min_{0\le t\le n}
 \bigl\{d_\pi(p,t)+d_u(t,q)\bigr\}.
\end{equation}
In particular,
\[
 d_\pi^{(m+1)}(p,q)
 =\min_{0\le t\le n}
 \bigl\{d_\pi(p,t)+d_\pi^{(m)}(t,q)\bigr\}.
\]
\end{lemma}

\begin{proof}
For the identity permutation \(\iota\),
\(d_\iota(p,t)=\max\{t-p,0\}\).  Hence the right side of
\eqref{eq:rectangle-rank-product}, with \(\pi=\iota\), equals
\(d_u(p,q)\).  Indeed, choosing \(t=p\) gives equality.  If \(t<p\),
then \(d_u(t,q)\ge d_u(p,q)\); if \(t>p\), deleting \(t-p\) rows
removes at most \(t-p\) counted entries, and
\(d_\iota(p,t)=t-p\) pays for that possible loss.

Now build \(\pi\) by a shortest sequence of adjacent exchanges.  Suppose
the part already built is \(\rho\), put \(\theta=S_\rho(u)\), and append
an exchange at \(i\).  It swaps the increasing pair
\(a=\rho_i<b=\rho_{i+1}\).  Fix \(q\) and write
\[
 A_t=d_\rho(i,t)+d_u(t,q),
 \qquad M=\min_t A_t.
\]
The adjacent-row identities for \(d_\rho\) give
\[
 d_\theta(i-1,q)=\min_t\bigl(A_t+\mathbf1_{\{a\le t\}}\bigr),
 \qquad
 d_\theta(i+1,q)=\min_t\bigl(A_t-\mathbf1_{\{b\le t\}}\bigr).
\]
All quantities are integers.  Comparing these two minima with
\(d_\theta(i,q)=M\) shows that
\[
 \theta_i\le q
 \quad\Longleftrightarrow\quad
 \text{every minimizing cut satisfies }t\ge a,
\]
whereas
\[
 \theta_{i+1}\le q
 \quad\Longleftrightarrow\quad
 \text{some minimizing cut satisfies }t\ge b.
\]

After exchanging \(a,b\), the rank \(d_\rho(i,t)\) increases by one exactly
for \(a\le t<b\); all other rows are unchanged.  The new minimum is therefore
\(M+1\) exactly when every old minimizing cut lies in \([a,b)\).  By the two
preceding equivalences, this is exactly the condition
\(\theta_i\le q<\theta_{i+1}\).  Hence the minimum changes on precisely the
same strip as the rank table obtained by applying the adjacent exchange at
\(i\) to \(\theta\).  Induction over the shortest sequence producing \(\pi\)
proves the formula.
\end{proof}

Formula \eqref{eq:rectangle-rank-product} is also the finite-permutation
case of \cite[Equation~(2)]{Pflueger}; the proof above is self-contained.
In standard terminology, the three adjacent identities above are the
\(0\)-Hecke relations, their induced product on permutations is the Demazure
product, and comparison of the rectangle-rank tables is the rank-matrix form
of strong Bruhat order; see \cite{BjornerBrenti,Pflueger}.  Only the explicit
identities and formulas proved here are used below.

\begin{theorem}
\label{thm:layered-rank-criterion}
For \(1\le r<n\) and \(m\ge1\),
\[
 t_r(Q)\le m
\]
if and only if
\begin{equation}
\label{eq:layered-rank-cut}
 \sum_{\alpha=1}^{m}d_Q(c_{\alpha-1},c_\alpha)\ge r
\end{equation}
for every \(c_1,\ldots,c_{m-1}\in\{0,\ldots,n\}\), where
\[
 c_0=n-r,\qquad c_m=r.
\]
In particular, all one-sided thresholds depend on \(Q\) only through the
ordinary permutation matrix \(P_Q\).
\end{theorem}

\begin{proof}
Put \(\pi=\sigma_Q\).  Theorem~\ref{thm:label-sorting} gives
\[
 t_r(Q)\le m
 \quad\Longleftrightarrow\quad
 d_\pi^{(m)}(n-r,r)=r;
\]
the right side says that all \(r\) labels in \([r]\) occupy the last
\(r\) positions.  Iterating Lemma~\ref{lem:rectangle-rank-product} gives
\[
 d_\pi^{(m)}(c_0,c_m)
 =\min_{c_1,\ldots,c_{m-1}}
 \sum_{\alpha=1}^m d_Q(c_{\alpha-1},c_\alpha).
\]
With \(c_0=n-r\) and \(c_m=r\), the left side is at most \(r\).
It equals \(r\) exactly when every sum under the minimum is at least
\(r\), which is precisely \eqref{eq:layered-rank-cut}.
\end{proof}

The recurrence gives the monotonicity needed below.
If every southwest rectangle of \(\rho\) contains at least as many dots as
the corresponding rectangle of \(\pi\), induction in \(m\) gives the same
comparison after every repetition.  Hence
\begin{equation}
\label{eq:rank-order-monotonicity}
 d_\pi(p,q)\le d_\rho(p,q)\quad(0\le p,q\le n)
 \quad\Longrightarrow\quad
 t_r(\rho)\le t_r(\pi)
 \qquad(1\le r<n).
\end{equation}

For an oscillatory matrix \(A=LDU\), let \(Q_+\) and \(Q_-\) be the
positive upper-factor sequences of \(U\) and \(L^T\).  Lemma
\ref{lem:greedy-reachability} and the corner formula give
\begin{equation}
\label{eq:word-profile-formula}
 e_k(A)=\max\{t_k(Q_+),t_k(Q_-)\},
 \qquad1\le k<n,
\end{equation}
and \(e_n(A)=1\).  Thus the profile depends on the ordered locations of the
positive factors, not on their positive numerical values.

There is also a finite realization statement.  Put
\[
 \Omega_n=(n-1)(n-2,n-1)\cdots(1,2,\ldots,n-1),
\]
and let \(\mathcal W_n\) be the complete subwords of \(\Omega_n\).  This is
the fixed staircase order used in successive elementary bidiagonal
factorization \cite{JohnsonOleskyDriessche,FallatJohnson}.

\begin{theorem}
\label{thm:finite-word-realization}
A positive integer vector \((f_1,\ldots,f_n)\) is the minor-order exponent
profile of an \(n\)-by-\(n\) oscillatory matrix if and only if \(f_n=1\)
and there exist \(Q_+,Q_-\in\mathcal W_n\) such that
\[
 f_k=\max\{t_k(Q_+),t_k(Q_-)\},
 \qquad1\le k<n.
\]
\end{theorem}

\begin{proof}
Ordinary Neville elimination in the fixed staircase order \(\Omega_n\)
makes the positive upper factors a subword \(Q_+\), and similarly gives
\(Q_-\) from \(L^T\).  The oscillatory criterion says that both are
complete.  Formula~\eqref{eq:word-profile-formula} proves necessity.

Conversely, for \(Q_+,Q_-\in\mathcal W_n\), put
\[
 U(Q)=\prod_{q\text{ in }Q}(I+E_{q,q+1}),
 \qquad A=U(Q_-)^TU(Q_+).
\]
This is nonsingular TN, and completeness makes both first off-diagonals
positive.  Hence \(A\) is oscillatory, and
\eqref{eq:word-profile-formula} gives the prescribed vector.  All factors
have integer entries and determinant one, so \(A\) is an integer
unimodular realization.  Thus every realizable profile has such a
realization.
\end{proof}

\section{Compatibility beyond pairwise bounds}
\label{sec:compatibility}

The preceding inequalities compare two coordinates at a time.  Positivity
after two passes at one order in fact forces positivity after three
passes two orders higher.  Only the rectangle counts defined in the
preceding section are needed.

\subsection{The first two-pass implication}

Fix \(k\), put \(a=n-k\), write \(\sigma=\sigma_Q\), and abbreviate
\(d_Q\) in \eqref{eq:rank-rectangle} to \(d\).  Theorem
\ref{thm:layered-rank-criterion}, with two and three factors, gives
\begin{align}
 t_k(Q)\le2
 &\quad\Longleftrightarrow\quad
 P_t:=d(a,t)+d(t,k)\ge k
 \quad(0\le t\le n),
 \label{eq:two-link-rank}\\
 t_{k+2}(Q)\le3
 &\quad\Longleftrightarrow\quad
 C(r,s):=d(a-2,r)+d(r,s)+d(s,k+2)\ge k+2
 \quad(0\le r,s\le n).
 \label{eq:three-link-rank}
\end{align}
Thus only the ranks of southwest rectangles of \(P_Q\) remain.

We shall use the following elementary triangle inequality for these ranks:
\begin{equation}
\label{eq:rectangle-rank-triangle}
 d(p,q)+d(q,r)\ge d(p,r)
 \qquad(0\le p,q,r\le n).
\end{equation}
If \(q<p\), the second term already contains the first rectangle together
with some additional rows.  If \(q\ge p\), deleting rows
\(p+1,\ldots,q\) removes at most \(q-p\) entries from \(d(p,r)\), while
among the \(q\) values at most \(q\), at most \(p\) can occur in the first
\(p\) rows; hence \(d(p,q)\ge q-p\).  This proves
\eqref{eq:rectangle-rank-triangle} in both cases.

\begin{lemma}
\label{lem:two-row-augmentation}
Assume \(k\ge3\), \(a\ge4\), and \(P_t\ge k\) for every \(t\).  Then
\[
 C(r,s)\ge\min\{P_r,P_s\}+2
 \qquad(0\le r,s\le n).
\]
\end{lemma}

Put
\begin{align}
 x&=d(a-2,r)-d(a,r),\nonumber\\
 y&=d(s,k+2)-d(s,k),\nonumber\\
 z&=d(r,s)+d(s,k)-d(r,k),\label{eq:rank-gaps}\\
 u&=d(a-2,r)+d(r,s)-d(a,s).\nonumber
\end{align}
All four quantities are nonnegative.  The first two count entries in the
two newly exposed rows and columns; the last two are instances of
\eqref{eq:rectangle-rank-triangle}.  Direct expansion gives
\begin{equation}
\label{eq:rank-gap-identities}
 C(r,s)-P_r=x+y+z,
 \qquad C(r,s)-P_s=u+y.
\end{equation}

These identities give the required estimate when \(s\le k\), when
\(r\ge a\), when \(r<a\) and \(k<s\le a\), and when
\(k\le r<a\) and \(s>k\).  The only remaining range is
\[
 r<\min\{a,k\},\qquad s>\max\{a,k\}.
\]

\begin{lemma}
\label{lem:crossing-rectangle}
Assume \(k\ge3\) and \(a\ge4\).  Suppose
\[
 r<\min\{a,k\},\qquad s>\max\{a,k\},\qquad P_s\ge k.
\]
Then
\[
 P_r\le P_s\Longrightarrow z\ge2,
 \qquad
 P_s\le P_r\Longrightarrow u+y\ge2.
\]
\end{lemma}

\begin{proof}
Put \(m=s-k\) and \(t=d(s,k)\).  Thus \(m\ge1\), and \(t\) is the
number of permutation-matrix entries in columns at most \(k\) below row
\(s\).

First suppose \(P_r\le P_s\).  Let
\begin{align*}
 b&=\#\{i\le r:k<\sigma_i\le s\},&
 q&=\#\{i\le a:\sigma_i\le k\},\\
 M&=\#\{i>a:k<\sigma_i\le s\},&
 c&=\#\{i\le r:\sigma_i\le k\},
\end{align*}
and let \(v\) be the number of values in \([r]\) occurring in the first
\(a\) positions.  Counting the indicated rectangles gives
\[
 z=m-b+t,\qquad P_s-k=M+t-q,\qquad P_r=k+r-v-c.
\]
Since \(M\le m-b\), the inequality \(P_s\ge k\) gives \(q\le z\).
Also \(v,c\le q\).  The inequality \(P_r\le P_s\) therefore implies
\begin{equation}
\label{eq:crossing-r-bound}
 r\le v+c-q+M+t\le q+z\le2z.
\end{equation}
Moreover,
\begin{equation}
\label{eq:crossing-z-and-t}
 z\ge m-r+t,
 \qquad t\ge a-m-q.
\end{equation}
The second inequality follows because the \(k-q\) small values below row
\(a\) have only \(s-a\) available positions before row \(s\).

If \(z=0\), then \eqref{eq:crossing-r-bound} gives \(r=0\), whereas
\(z=m-b+t\ge m\), a contradiction.  If \(z=1\),
\eqref{eq:crossing-r-bound}--\eqref{eq:crossing-z-and-t} force
\[
 a=4,\qquad r=2,\qquad q=1,\qquad m+t=3.
\]
Then \(z=m-b+t=1\) gives \(b=2\), hence \(c=0\).  But the first
inequality in \eqref{eq:crossing-r-bound}, with \(v\le1\) and
\(M+t\le z=1\), requires \(c=1\).  Thus \(z\ge2\).

Now suppose \(P_s\le P_r\), and assume \(u+y\le1\).  Write
\[
 d_0=d(a-2,r),\qquad
 e=d(r,s)-d(a,s)=\#\{r<i\le a:\sigma_i\le s\},
\]
so \(u=d_0+e\).  Put
\[
 \xi=d(a-2,r)-d(a,r),\qquad
 c=\#\{i\le r:\sigma_i\le k\}.
\]
Since \(P_s\ge k\) and \(P_s\le P_r\),
\begin{equation}
\label{eq:crossing-pr-defect}
 0\le P_r-k=d_0-\xi-c.
\end{equation}

If \(d_0=0\), then \(c=0\) and \(P_r=P_s=k\).  All values in \([r]\)
occur in the first \(a-2\) positions but none in the first \(r\), so
\(e\ge r\) and hence \(r\le1\).  If \(r=0\), then
\(e=s-d(a,s)=m+t\); the inequality \(e+y\le1\) would put the \(k+2\)
values in \([k+2]\) into only \(k+1\) positions.  If \(r=1\), then
\(e=1\), \(y=0\), and only one small value occurs in the first \(a\)
positions.  At least \(m-1\) of the values in \((k,s]\) then occur below
row \(a\), while \(t\ge a-m-1\).  The equality \(P_s=k\) would give
\[
 1\ge(m-1)+(a-m-1)=a-2\ge2,
\]
a contradiction.

It remains that \(d_0=1\), so \(e=y=0\).  All small values in the first
\(a\) positions already occur in the first \(r\), and hence \(q=c\).
If \(b\) again counts the values in \((k,s]\) in the first \(r\)
positions, then comparison of \(P_s-k\) and \(P_r-k\) gives
\[
 m-b+t+\xi\le1.
\]
Among the values in \([r]\), exactly \(r-1+\xi\) occur in the first
\(a\) positions.  Hence \(b\le r-c\le1-\xi\), and therefore
\(m+t\le2\).  Finally \(t\ge a-m-c\) and
\eqref{eq:crossing-pr-defect} gives \(c\le1\), so
\[
 a\le m+t+c\le3,
\]
contrary to \(a\ge4\).  Thus \(u+y\ge2\).
\end{proof}

\begin{proof}[Proof of Lemma~\ref{lem:two-row-augmentation}]
If \(s\le k\), then
\[
 y+z=d(r,s)+d(s,k+2)-d(r,k)\ge2,
\]
because \(d(s,k+2)\ge k+2-s\) while
\(d(r,k)-d(r,s)\le k-s\).  Hence \(C\ge P_r+2\).
If \(r\ge a\), then
\[
 u=d(a-2,r)-\bigl(d(a,s)-d(r,s)\bigr)\ge2,
\]
because the two terms are bounded below and above by \(r-a+2\) and
\(r-a\), respectively.  Hence \(C\ge P_s+2\).

Assume now \(r<a\) and \(s>k\).  If \(s\le a\), splitting the entries
below row \(a\) according as their columns are at most \(k\) or in
\((k,s]\) gives
\[
 d(a,s)\le \bigl(d(r,s)-d(r,k)\bigr)+d(s,k)=z,
 \qquad d(s,k)\le z.
\]
Thus \(P_s=d(a,s)+d(s,k)\le2z\).  Since \(P_s\ge k\ge3\), integrality gives
\(z\ge2\), so \(C\ge P_r+2\).  If \(r\ge k\), the analogous row split
gives
\[
 d(a,r)\le d(a-2,r)\le u,
 \qquad
 d(r,k)\le d(r,s)-d(a,s)+d(a,r)\le u.
\]
Hence \(P_r\le2u\), so \(u\ge2\) and \(C\ge P_s+2\).

The remaining region is exactly that of
Lemma~\ref{lem:crossing-rectangle}.  According as \(P_r\le P_s\) or
\(P_s\le P_r\), that lemma and \eqref{eq:rank-gap-identities} give
\(C\ge P_r+2\) or \(C\ge P_s+2\).  This proves the result.
\end{proof}

\begin{theorem}
\label{thm:two-sweep-compatibility}
For every adjacent factor sequence \(Q\),
\[
 t_k(Q)\le2\quad\Longrightarrow\quad t_{k+2}(Q)\le3,
 \qquad3\le k\le n-4.
\]
Consequently, for every nonsingular TN matrix \(B\),
\begin{equation}
\label{eq:one-sided-matrix-two-sweep}
 \Delta_{L_k,R_k}(B^2)>0
 \quad\Longrightarrow\quad
 \Delta_{L_{k+2},R_{k+2}}(B^3)>0,
 \qquad3\le k\le n-4.
\end{equation}
For every oscillatory matrix \(A\),
\[
 e_k(A)\le2\quad\Longrightarrow\quad e_{k+2}(A)\le3,
 \qquad3\le k\le n-4.
\]
By checkerboard inversion, also
\[
 e_{k+2}(A)\le2\quad\Longrightarrow\quad e_k(A)\le3,
 \qquad2\le k\le n-5.
\]
\end{theorem}

\begin{proof}
The hypothesis and \eqref{eq:two-link-rank} give \(P_t\ge k\) for all
\(t\).  Lemma~\ref{lem:two-row-augmentation} gives
\[
 C(r,s)\ge\min\{P_r,P_s\}+2\ge k+2,
\]
so \eqref{eq:three-link-rank} yields \(t_{k+2}(Q)\le3\).

For \eqref{eq:one-sided-matrix-two-sweep}, apply the one-sided equality in
Lemma~\ref{lem:greedy-reachability} to the upper sequence in an adjacent
factorization of \(B\).  For an oscillatory matrix, apply the sequence
result separately to \(Q_+\) and \(Q_-\), then use
\eqref{eq:word-profile-formula}.  The reverse-direction implication follows
from Theorem~\ref{thm:checkerboard-inverse-duality}.
\end{proof}

\begin{corollary}
\label{cor:pairwise-bounds-not-sufficient}
The universal, neighboring-order, and all-order inequalities do not
characterize exponent profiles.
\end{corollary}

\begin{proof}
For \(n=7\), the vector
\[
 (1,2,2,3,4,3,1)
\]
satisfies all three families of pairwise inequalities.  It cannot be an
exponent profile, because \(e_3=2\) would imply \(e_5\le3\), contrary to
its fifth coordinate.
\end{proof}

The constant and the endpoint range are sharp.  Direct application of the
adjacent index maps to the following one-period permutations gives
\[
\begin{array}{c|c}
\sigma&(t_1,\ldots,t_6)\\ \hline
(2,4,5,6,3,7,1)&(1,2,2,3,3,4)\\
(2,3,4,5,6,7,1)&(1,2,3,4,5,6)\\
(2,5,6,4,1,7,3)&(2,3,2,2,3,4).
\end{array}
\]
The first has \(t_3=2,t_5=3\); the second rules out the lower endpoint
\(k=2\); the third has \(t_4=2,t_6=4\), showing that \(n-k\ge4\) is
necessary.

\subsection{Subadditivity and the order-three boundary}

\begin{lemma}
\label{lem:profile-subadditivity}
If \(a,b\ge1\) and \(a+b<n\), then
\[
 t_{a+b}(Q)\le t_a(Q)+t_b(Q).
\]
\end{lemma}

\begin{proof}
If either \(t_a(Q)\) or \(t_b(Q)\) is infinite, there is nothing to prove in
the extended positive integers.  Otherwise \(Q\) contains
every adjacent position: reaching a remote corner requires, somewhere in
its cell rectangle, every label \(1,\ldots,n-1\).  Write
\(Q=(q_1,\ldots,q_\ell)\), put
\(c=a+b\), \(H=n-c\), \(p=t_a(Q)\), and \(q=t_b(Q)\).
During the first \(p\) periods, assign to the cells
\((u,v)\in[a]\times[H]\) of \({\cal R}_c\) the completion numbers of the
cells \((u,b+v)\) of \({\cal R}_a\).  Their adjacent-position labels agree:
\[
 c-u+v=a-u+(b+v).
\]
This fills the first \(a\) rows of \({\cal R}_c\).  Delays caused by the
omitted first \(b\) columns of \({\cal R}_a\) do not violate any target
dependency.

During the next \(q\) periods, assign to cell \((a+u,v)\) the completion
number of \((u,v)\in[b]\times[H]\) in \({\cal R}_b\), shifted by \(p\ell\).
Again the labels agree,
\[
 c-(a+u)+v=b-u+v.
\]
All dependencies within this second block are preserved, and its top
boundary is already filled after the first phase.  Thus every cell of
\({\cal R}_c\) is added within \(p+q\) periods.
\end{proof}

The remaining boundary case is \(k=3\).  We reduce the two triples at the
ends of a one-pass list to four extremal forms and then follow the last
positions in three fixed lists.

For a permutation \(\pi=(\pi_1,\ldots,\pi_n)\), let
\[
 X_r(\pi)=\{\pi_{n-r+1},\ldots,\pi_n\},
 \qquad
 Y_r(\pi)=\{i:\pi_i\le r\},
\]
with both sets written increasingly.

Permutation products below are composed from right to left:
\((pq)(i)=p(q(i))\).

\begin{lemma}
\label{lem:two-block-subset-criterion}
For permutations \(\pi,\rho\), two successive adjacent maps satisfy
\[
 S_\rho\bigl(S_\pi(L_r)\bigr)=R_r
 \quad\Longleftrightarrow\quad
 X_r(\rho)\preceq Y_r(\pi).
\]
In particular,
\[
 t_r(Q)\le2
 \quad\Longleftrightarrow\quad
 X_r(\sigma_Q)\preceq Y_r(\sigma_Q).
\]
\end{lemma}

\begin{proof}
Theorem~\ref{thm:label-sorting} and
Lemma~\ref{lem:rectangle-rank-product} say that the left side is equivalent
to
\[
 |X_r(\rho)\cap[t]|+|Y_r(\pi)\setminus[t]|\ge r
 \qquad(0\le t\le n).
\]
After cancelling \(r=|Y_r(\pi)|\), this becomes
\[
 |X_r(\rho)\cap[t]|\ge|Y_r(\pi)\cap[t]|
 \qquad(0\le t\le n),
\]
which is exactly componentwise comparison of the two increasing sets.
\end{proof}

\begin{lemma}
\label{lem:canonical-triple-reduction}
Let \(n\ge7\), and let \(\pi\) be a permutation of \([n]\) for which
\[
 I:=X_3(\pi)\preceq J:=Y_3(\pi).
\]
Put
\[
 L=[3],\qquad T=\{n-2,n-1,n\},\qquad M=\{4,\ldots,n-3\},
\]
and \(r=|I\cap L|=|J\cap T|\).  There is a permutation \(g\) such that
\[
 d_g(p,q)\le d_\pi(p,q)\qquad(0\le p,q\le n)
\]
and whose boundary triples have one of the following forms:
\[
\begin{array}{c|c}
r&(X_3(g),Y_3(g))\\ \hline
0&(K,K),\quad K\subset M,\ |K|=3,\\
1&(\{a,b,c\},\{a,b,c\}),\quad a\in L,\ b\in M,\ c\in T,\\
2&(\{a,3,c\},\{a,n-2,c\}),\quad
   a\in\{1,2\},\ c\in\{n-1,n\},\\
3&(L,T).
\end{array}
\]
In the first three cases let \(p\) be the list of the elements of
\(X_3(g)^c\), followed by those of \(X_3(g)\), each block increasing, and
write \(g=pz\) in the preceding composition convention.  Then the shortest
adjacent sequences for \(p,z\) concatenate without cancellation in both
orders, and \(zp=\rho_r\), where
\[
\begin{array}{c|c}
0&\rho_0=(4,5,\ldots,n,1,2,3),\\
1&\rho_1=(3,4,\ldots,n-3,n-1,n,1,2,n-2),\\
2&\rho_2=(2,4,5,\ldots,n-2,n,1,3,n-1).
\end{array}
\]
Moreover, in case \(r=1\), \(p_{n-2}=a\le3\); in case \(r=2\),
\(p_{n-2}=a\le2\) and \(p_{n-1}=3\).
\end{lemma}

\begin{proof}
The equality defining \(r\) holds because both sides count the dots of the
permutation matrix in the last-three-rows by first-three-columns rectangle.
For fixed \(I,J\), split the positions into
\[
 T\cap J,\quad T\setminus J,\quad J\setminus T,
 \quad[n]\setminus(T\cup J),
\]
and the values into
\[
 I\cap L,\quad I\setminus L,\quad L\setminus I,
 \quad[n]\setminus(I\cup L).
\]
Corresponding classes have equal sizes.  Match them increasingly.  The
resulting permutation \(g_{I,J}\) has boundary sets \(I,J\).  Sorting within
one class only removes inverted pairs.  More explicitly, if positions
\(u<v\) contain \(\alpha>\beta\), exchanging them changes a southwest
count only when \(u\le p<v\) and \(\beta\le q<\alpha\), and then lowers it
by one.  Thus none of the counts increases.
Thus \(d_{g_{I,J}}\le d_\pi\).

Write \(I=(i_1,i_2,i_3)\) and \(J=(j_1,j_2,j_3)\).  Keeping \(r\) fixed,
raise a coordinate of \(I\) whenever the result is still a three-set below
\(J\).  In \(g_{I,J}\), the old value and the next value form an inverted
pair, so exchanging them and sorting within the four classes again cannot
increase a southwest count.  Similarly, a permitted lowering of a
coordinate of \(J\) exchanges two adjacent positions whose assigned values
are inverted, with the same effect.  Repeating these moves terminates.  The
largest possible \(I\) and then the smallest possible \(J\) are
\[
\begin{array}{c|c|c}
r&\text{largest }I&\text{smallest }J\\ \hline
0&J&I\\
1&(\min\{3,j_1\},j_2,j_3)&(i_1,i_2,\max\{n-2,i_3\})\\
2&(\min\{2,j_1\},3,j_3)&(i_1,n-2,\max\{n-1,i_3\})\\
3&L&T.
\end{array}
\]
These are exactly the four rows in the statement.

For \(r=0,1,2\), define \(p,z\) as stated.  Comparing inverted pairs within
and between the four increasing classes gives
\[
 \operatorname{inv}(pz)=\operatorname{inv}(p)+\operatorname{inv}(z),
 \qquad
 \operatorname{inv}(zp)=\operatorname{inv}(z)+\operatorname{inv}(p).
\]
Hence shortest adjacent sequences concatenate in both orders.  Substitution
in the same four classes gives the three displayed lists \(\rho_r=zp\).
Their last-position assertions follow from the definition of
\(p\).
\end{proof}

\begin{lemma}
\label{lem:three-model-tail-growth}
Let \(r\in\{0,1,2\}\), \(t\ge2\), and \(s=2t+1\le n-2\).  After \(t\)
repetitions of the adjacent map determined by \(\rho_r\),
\[
\begin{array}{c|c|c}
r&X_s(z_{\rho_r}^{(t)})&X_{s-1}(z_{\rho_r}^{(t)})\\ \hline
0&[s]&[s-1]\\
1&[s-1]\cup\{n-2\}&[s-2]\cup\{n-2\}\\
2&[s-2]\cup\{s,n-1\}&[s-2]\cup\{n-1\}.
\end{array}
\]
\end{lemma}

\begin{proof}
Write \(s_i\) for the adjacent exchange at positions \(i,i+1\).
Shortest sequences producing the three one-pass lists are
\begin{align*}
\rho_0&=\prod_{i=1}^{n-3}(s_{i+2}s_{i+1}s_i),\\
\rho_1&=\left(\prod_{i=1}^{n-5}s_{i+1}s_i\right)
 (s_{n-2}s_{n-3}s_{n-4})(s_{n-1}s_{n-2}s_{n-3}),\\
\rho_2&=s_1\left(\prod_{i=2}^{n-4}s_{i+1}s_i\right)
 (s_{n-1}s_{n-2}s_{n-3}),
\end{align*}
with \(i\) increasing.  Direct applications give the cases \(t=2,3\).
If the displayed block form holds at \(t\), one further application
compares only adjacent entries in its increasing intervals.  For example,
in the less symmetric third case, when \(t\ge3\), \(s=2t+1\), and
\(s+2\le n-2\), one \(\rho_2\)-pass sends
\[
 (s-1,s+1,\ldots,n-2,n,s,n-1,s-2,\ldots,1)
\]
to
\[
 (s+1,s+3,\ldots,n-2,n,s+2,n-1,s,\ldots,1),
\]
which is the same form with \(s\) replaced by \(s+2\).  The other two
cases are obtained by the same adjacent comparison.  Thus
\[
 z_{\rho_0}^{(t)}
 =(d+1,\ldots,n,d,d-1,\ldots,1),
 \qquad d=\min\{3t,n-1\},
\]
\[
 z_{\rho_1}^{(t)}
 =(s,s+1,\ldots,n-3,n-1,n,s-1,n-2,s-2,\ldots,1),
\]
and
\[
 z_{\rho_2}^{(2)}
 =(4,6,7,\ldots,n-2,n,5,n-1,2,1,3),
\]
\[
 z_{\rho_2}^{(t)}
 =(s-1,s+1,s+2,\ldots,n-2,n,s,n-1,s-2,\ldots,1)
 \quad(t\ge3).
\]
Empty intervals are omitted.  Reading the last \(s\) and \(s-1\) entries
gives the table.
\end{proof}

\begin{lemma}
\label{lem:order-three-cone}
If \(n\ge7\) and \(t_3(Q)\le2\), then
\[
 t_j(Q)\le\left\lfloor\frac j2\right\rfloor+1,
 \qquad3\le j\le n-2.
\]
\end{lemma}

\begin{proof}
Put \(\pi=\sigma_Q\).  Lemma~\ref{lem:two-block-subset-criterion} gives
\(X_3(\pi)\preceq Y_3(\pi)\).  Apply
Lemma~\ref{lem:canonical-triple-reduction}.  Since \(d_g\le d_\pi\),
\eqref{eq:rank-order-monotonicity} shows that it is enough to bound
\(t_j(g)\).

If \(r=3\), then \(g=(4,5,\ldots,n,1,2,3)\), a shortest adjacent sequence
for \(g\) is complete, and \(t_3(g)=1\).  The one-sided comparison proved
in Theorem~\ref{thm:all-order-comparison} gives
\[
 t_j(g)\le\left\lceil\frac j3\right\rceil
 \le\left\lfloor\frac j2\right\rfloor+1
 \qquad(j\ge3),
\]
which is already stronger than the required estimate.

Let \(r=0,1,2\), and choose shortest adjacent sequences \(P,Z\) producing
\(p,z\).  One block producing \(g\) is \(PZ\), while \(ZP\) produces
\(\rho_r\).  Hence \(m\) repeated \(g\)-blocks may be written, as an
ordinary concatenation,
\[
 P\,(ZP)^{m-1}\,Z.
\]

Take \(t\ge2\) and \(s=2t+1\le n-2\).  In case \(r=0\), the two suffix
sets in Lemma~\ref{lem:three-model-tail-growth} are componentwise smallest.
In case \(r=1\), the exceptional position \(n-2\) contains under \(p\) the
value \(a\le3\).  In case \(r=2\), positions \(n-2,n-1\) contain values at
most \(2,3\).  Therefore, in all three cases,
\[
 X_s(z_{\rho_r}^{(t)})\preceq Y_s(p),
 \qquad
 X_{s-1}(z_{\rho_r}^{(t)})\preceq Y_{s-1}(p).
\]
Lemma~\ref{lem:two-block-subset-criterion} says that the initial block
\(P\), followed by \(t\) blocks \(ZP\), carries both \(L_s\) and
\(L_{s-1}\) to their
rightmost sets.  The final block \(Z\) leaves a rightmost set fixed.
Consequently
\[
 t_s(g)\le t+1,\qquad t_{s-1}(g)\le t+1.
\]
This proves the estimate for every odd \(s\ge5\) and the preceding even
order.  Orders \(3,4\) follow from the hypothesis and the neighboring-order
inequality.  If \(n\) is even, the remaining even endpoint follows once
more from that inequality.
\end{proof}

\subsection{Bounds at the remaining orders}

\begin{theorem}
\label{thm:two-sweep-cone}
Let \(Q\) be any adjacent factor sequence.  If \(3\le k\le n-3\) and
\(t_k(Q)\le2\), then
\[
 t_j(Q)\le2+\left\lceil\frac{|j-k|}{2}\right\rceil,
 \qquad2\le j\le n-2.
\]
Consequently, every oscillatory matrix satisfies
\[
 e_k(A)\le2
 \quad\Longrightarrow\quad
 e_j(A)\le2+\left\lceil\frac{|j-k|}{2}\right\rceil,
 \qquad2\le j\le n-2.
\]
\end{theorem}

\begin{proof}
First let \(j\ge k\).  For \(k=3\), the result is
Lemma~\ref{lem:order-three-cone}; the short case \(n=6\) follows directly
from the neighboring-order inequality.  Let \(k\ge4\), and write
\(j=(q+1)k+s\), where \(q\ge0\) and \(0\le s<k\).  The neighboring-order
	inequality, Theorem~\ref{thm:two-sweep-compatibility}, and the all-order
	comparison give
\[
 t_{k+s}(Q)\le
 \begin{cases}
 2,&s=0,\\
 3,&s=1,2,\\
 4,&3\le s<k.
 \end{cases}
\]
Repeated subadditivity therefore gives
\[
 t_j(Q)\le q\,t_k(Q)+t_{k+s}(Q)
 \le2+\left\lceil\frac{qk+s}{2}\right\rceil.
\]
For the last inequality use \(qk\ge4q\) and the four cases above.  This is
the required right-hand bound.

For \(j\le k\), reflect positions and complement index sets.  The reflected
sequence \(Q^*=(n-q_1,\ldots,n-q_\ell)\) satisfies
\(t_r(Q)=t_{n-r}(Q^*)\); apply the already proved right-hand bound with
reference order \(n-k\).  Finally, if \(e_k(A)\le2\), both one-sided
sequences \(Q_+,Q_-\) satisfy the hypothesis.  Take their pointwise maximum
in \eqref{eq:word-profile-formula}.
\end{proof}

\section{Conclusion}

The positive support of every compound row of a nonsingular TN matrix is a
complete componentwise interval, and its two endpoints compose under
multiplication.  This determines every positive minor of every power and
gives the fixed-order corner formula, the sharp fixed-order bound, and the
comparison inequalities between different minor orders.  Checkerboard
inversion reverses the first \(n-1\) components of the profile and leaves
the determinant exponent equal to one.

The ordered locations of the positive factors in an adjacent bidiagonal
factorization determine the profile; their positive values do not.  Two
permutations suffice, and every realizable profile has an integer unimodular
realization.  The pairwise inequalities do not characterize the profiles:
if \(3\le k\le n-3\) and \(e_k(A)\le2\), then
\[
 e_j(A)\le2+\left\lceil\frac{|j-k|}{2}\right\rceil
 \qquad(2\le j\le n-2).
\]
The first implication in this family also holds for one remote corner of an
arbitrary nonsingular TN matrix.

\end{document}